\documentclass[a4paper,12pt]{article}

\usepackage{amsmath}
\usepackage{amssymb}
\usepackage{amsthm}
\usepackage{enumerate}
\usepackage{graphicx}
\usepackage[hmargin=2.5cm,vmargin=2.5cm]{geometry}
\usepackage{parskip}
\usepackage{color}
\usepackage{lineno}
\usepackage[nolists,nomarkers]{endfloat}

\newtheorem{lemma}{Lemma}[section]

\newtheorem{theorem}{Theorem}[section]

\DeclareMathOperator{\prob}{\mathbb{P}}
\DeclareMathOperator{\bigcomp}{\bigcirc}
\DeclareMathOperator{\rec}{flip}
\DeclareMathOperator{\Leb}{Leb}

\newcommand{\symmdiff}{\,\triangle\,}

\author{Cristian Mitrea\footnote{Bernoulli Institute for Mathematics, Computer Science and Artificial Intelligence, University of Groningen, PO Box 407, 9700 AK Groningen, The Netherlands. E-mail: c.mitrea@student.rug.nl, cristian.mitrea@icloud.com, a.e.sterk@rug.nl.} \and Alef E.~Sterk$^*$}
\title{Singular functions obtained via \\ random function iteration}
\date{\today}

\begin{document}

\maketitle

\begin{abstract}
In this paper we consider a discrete-time dynamical system on the real line by random iteration of two functions. These functions are assumed to satisfy appropriate monotonicity conditions; optionally, a symmetry condition may be imposed. Using Bernoulli measures on the space of binary sequences we show that sequences generated by the iteration process almost surely diverge to either plus or minus infinity. The function that assigns to each initial point the probability that the iterates diverge to plus infinity is shown to satisfy a functional equation that encodes self-similarity properties. In this way we obtain singular functions that are well-known from the literature: Cantor-like functions, Lebesgue singular functions, and the Minkowski question mark function.
\end{abstract}






\newpage

\tableofcontents


\newpage


\section{Introduction}
\label{sec:intro}

The Cantor function, originally introduced to serve as a counter example in attempts to prove extensions of the Fundamental Theorem of Calculus, is a standard example of a singular function. One way to define the Cantor function is by means of ternary expansions: write $x \in [0,1]$ as
\[
x = \sum_{k=1}^\infty \frac{a_k(x)}{3^k}, \quad a_k(x) \in \{0,1,2\},
\]
and denote by $n(x)$ the smallest positive integer $k$ for which $a_{k}(x)=1$, provided such an integer exists; otherwise set $n(x)=\infty$. Then the Cantor function is given by
\begin{equation}
\label{eq:cantor_function}
C : [0,1] \to \mathbb{R}, \quad
C(x) = \frac{1}{2^{n(x)}} + \frac{1}{2}\sum_{k=1}^{n(x)-1} \frac{a_k(x)}{2^k}.
\end{equation}
This definition does not depend on the choice of the ternary expansion of $x$. The graph of the Cantor function, often called the Devil's Staircase, is frequently discussed in works on fractal geometry \cite{Berger:2001, Helmberg:2007, PJS:2004}. A detailed review of the Cantor function and its properties can be found in~\cite{Dovgoshey:2006}.

Stankewitz and Rolf \cite{StankewitzRolf:2012} suggested that the Cantor function can also be obtained by means of random function iteration as follows. Consider the functions $f_0(x) = 3x$ and $f_1(x) = 3x-2$. Let $(\omega_1,\omega_2,\omega_3,\dots)$ be a sequence of $0$'s and $1$'s, where each entry $\omega_i$ is chosen independently at random and with equal probability. For any $x_0 \in [0,1]$ consider the sequence
\[
x_{n+1} = f_{\omega_{n+1}}(x_n), \quad n\geq 0.
\]
The graph of the function that assigns to each $x_0$ the probability that $x_n \to \infty$ resembles the graph of the Cantor function.  Inspired by this work, Armstrong and Schaubroeck~\cite{ArmstrongSchaubroeck:2020} explore random function iteration using finitely many linear functions. However, neither of the works \cite{ArmstrongSchaubroeck:2020, StankewitzRolf:2012} provides a proof that the classical Cantor function is indeed obtained in this manner.

In this paper we use the Bernoulli measure on the space of binary sequences to prove that the random function iteration approach in \cite{StankewitzRolf:2012} indeed gives rise to the Cantor function. In fact, by considering a class of functions $f_0$ and $f_1$ that satisfy suitable monotonicity assumptions we obtain several singular functions that are well-known from the literature: Cantor-like functions, Lebesgue singular functions, and variants of the Minkowksi question mark function. The work presented in this paper is based on a Bachelor research project performed by the first author \cite{Mitrea:2025} under supervision of the second author.


\newpage

\section{Random function iteration}
\label{sec:iteration}

As a first step we define a suitable probability measure on the space of all binary sequences. In this way we define a function that assigns to each initial point the probability that the sequence of the iterates diverges to plus infinity. Throughout the paper we will use the notation $\mathbb{N} = \{1,2,3,\dots\}$ and $\mathbb{N}_0 = \{0,1,2,3,\dots\}$. 


\subsection{Binary sequences and probability measures}

We consider the set of all binary sequences given by $\Omega = \{0,1\}^\infty$. An element $\omega \in \Omega$ will be written as $\omega = (\omega_1,\omega_2,\omega_3,\dots)$. In addition, we denote the flipped sequence, where all $0$'s become $1$'s and vice versa, as
\[
\overline{\omega} = (1-\omega_1,1-\omega_2,1-\omega_3,\dots),
\]
and for any $A \subseteq \Omega$ we define
\[
\rec(A) = \{\overline{\omega} \,:\, \omega \in A\}.
\]

We briefly sketch how probability measures can be defined on $\Omega$. For any $n \in \mathbb{N}$ we define a so-called cylinder set of length $n$ by
\[
[\alpha_1,\dots,\alpha_n] = \{ \omega \in \Omega \,:\, \omega_i = \alpha_i \text{ for all } i=1,\dots,n \}.
\]
For a fixed number $0 < p < 1$ set $q = 1-p$ and define
\[
\prob_p([\alpha_1,\dots,\alpha_n]) = p^{n-\sum_{i=1}^n \alpha_i} q^{\sum_{i=1}^n \alpha_i}.
\]
From this definition it follows that for any $1 < k < n$ we have
\begin{equation}
\label{eq:cylinder_split}
\prob_p([\alpha_1,\dots,\alpha_n])
	= \prob_p([\alpha_1,\dots,\alpha_k])\prob_p([\alpha_{k+1},\dots,\alpha_n]).
\end{equation}
Then $\prob_p$ can be extended to a unique probability measure, known as the Bernoulli measure, on the $\sigma$-algebra generated by the collection of all cylinder sets; see, for example, \cite{Berger:2001, DajaniKalle:2021} for further details.

The one-sided Bernoulli shift is defined as
\[
\sigma : \Omega \to \Omega, \quad (\omega_1,\omega_2,\omega_3,\dots) \mapsto (\omega_2,\omega_3,\omega_4,\dots).
\]
A standard result from ergodic theory \cite{Berger:2001, DajaniKalle:2021} is that $\sigma$ preserves the measure $\prob_p$: for all measurable sets $A \subseteq \Omega$ we have
\[
\prob_p(\sigma^{-1}(A)) = \prob_p(A).
\]
The following additional properties will be useful in the remainder of the paper:

\begin{lemma}
\label{lemma:sigma1}
For any measurable set $A \subseteq \Omega$ we have:
\begin{enumerate}[(i)]
\item $\prob_p(\sigma^{-1}(A) \cap [0]) = p\prob_p(A)$; 
\item $\prob_p(\sigma^{-1}(A) \cap [1]) = q\prob_p(A)$;
\item $\prob_p(\rec(A)) = \prob_q(A)$.
\end{enumerate}
\end{lemma}

\begin{proof}
(i) For any cylinder set we have
\[
\begin{split}
\sigma^{-1}([\alpha_1,\dots,\alpha_n]) \cap [0]
	& = ([0,\alpha_1,\dots,\alpha_n] \cup [1,\alpha_1,\dots,\alpha_n]) \cap [0] \\
	& = [0,\alpha_1,\dots,\alpha_n],
\end{split}
\]
so the observation in equation \eqref{eq:cylinder_split} gives
\[
\prob_p(\sigma^{-1}([\alpha_1,\dots,\alpha_n]) \cap [0]) = p\prob_p([\alpha_1,\dots,\alpha_n]).
\]
It easily follows that the equality also holds for any finite union of distinct cylinder sets.

Let $\symmdiff$ denote the symmetric difference of two sets:
\[
A \symmdiff B = (A \setminus B) \cup (B \setminus A) = (A \cup B) \setminus (A \cap B).
\]
For all measurable sets $A, B \subseteq \Omega$ we have
\[
\begin{split}
|\prob_p(A) - \prob_p(B)|
	& = |(\prob_p(A \setminus B) + \prob_p(A \cap B)) - (\prob_p(B \setminus A) + \prob_p(B \cap A))| \\
	& = |\prob_p(A \setminus B) - \prob_p(B \setminus A)| \\
	& \leq \prob_p(A \setminus B) + \prob_p(B \setminus A) \\
	& = \prob_p(A \symmdiff B),
\end{split}
\]
and
\[
\begin{split}
|\prob_p(\sigma^{-1}(A) \cap [0]) - \prob_p(\sigma^{-1}(B) \cap [0])|
	& \leq \prob_p((\sigma^{-1}(A) \cap [0]) \symmdiff (\sigma^{-1}(B) \cap [0])) \\
	& = \prob_p((\sigma^{-1}(A) \symmdiff \sigma^{-1}(B)) \cap [0]) \\
	& = \prob_p(\sigma^{-1}(A \symmdiff B) \cap [0]) \\
	& \leq \prob_p(\sigma^{-1}(A \symmdiff B)) \\
	& = \prob_p(A \symmdiff B).
\end{split}
\]
For any $\varepsilon>0$ there exists a set $U \subseteq \Omega$ such that $U$ is a finite union of distinct cylinder sets and
\[
\prob_p(A \symmdiff U) \leq \varepsilon.
\]
Since $\prob_p(\sigma^{-1}(U) \cap [0]) = p\prob_p(U)$ we have
\[
\begin{split}
|\prob_p(\sigma^{-1}(A) \cap [0]) - p\prob_p(A)|
	& \leq |\prob_p(\sigma^{-1}(A) \cap [0]) - \prob_p(\sigma^{-1}(U) \cap [0])| \\
	& \hspace{10mm} + p|\prob_p(U) - \prob_p(A)| \\
	& \leq (1+p) \prob_p(A \symmdiff U) \\
	& \leq (1+p)\varepsilon.
\end{split}
\]
Since $\varepsilon>0$ is arbitrary, we conclude that statement (i) holds for any measurable set $A$. Statement (ii) follows in the same manner.

(iii) This follows from the fact that the maps $C \mapsto \prob_p(\rec(C))$ and $C \mapsto \prob_q(C)$, where $C$ is a cylinder set, are the same and extend to a unique probability measure on $\Omega$.
\end{proof}


\subsection{Divergent sequences}

\begin{figure}[ht]
\centering
\includegraphics[width=0.45\textwidth]{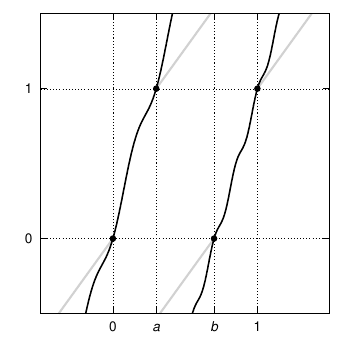}
\caption{Illustration of conditions (A1)--(A3). The gray lines have slope $s>1$.}
\label{fig:plot_f0f1}
\end{figure}

Now that that we have defined probability measures on the space of binary sequences we can study random function iteration. We consider functions $f_0, f_1 : \mathbb{R} \to \mathbb{R}$ that satisfy the following monotonicity assumptions:
\begin{itemize}
\item[(A1)] The functions $f_0$ and $f_1$ are continuous and strictly increasing.
\item[(A2)] There exist real numbers $0 < a \leq b < 1$ such that
\[
f_0(0) = 0, \quad
f_0(a) = 1, \quad
f_1(b) = 0, \quad
f_1(1) = 1.
\]
\item[(A3)] There exists a real number $s > 1$ such that the following implications hold:
\[
\begin{split}
x \leq 0 & \quad\Rightarrow\quad f_0(x) \leq s x, \\
x \geq a & \quad\Rightarrow\quad f_0(x) \geq 1+s(x-a), \\
x \leq b & \quad\Rightarrow\quad f_1(x) \leq s(x-b), \\ 
x \geq 1 & \quad\Rightarrow\quad f_1(x) \geq 1+s(x-1).
\end{split}
\]
See Figure \ref{fig:plot_f0f1} for an illustration. These inequalities ensure that iterates landing outside the unit interval will diverge to either plus or minus infinity.
\end{itemize}
Throughout the paper (A1)--(A3) will be assumed to hold without further mentioning. In some cases we need the following additional symmetry assumption:
\begin{itemize}
\item[(A4)] For all $x \in \mathbb{R}$ we have
\[
f_0(1-x) + f_1(x) = 1.
\]
Note that by replacing $x$ by $1-x$ we obtain the same identity in which the roles of $f_0$ and $f_1$ are interchanged. Moreover, we necessarily have $a=1-b$. Indeed, setting $x=b$ gives $f_0(1-b)=1$, but by assumption (A2) we have $f_0(a)=1$ as well. Since $f_0$ is injective by (A1) we have $a=1-b$.
\end{itemize}
Whenever (A4) is needed this will be explicitly stated.

For an initial point $x_0 \in \mathbb{R}$ and a symbol sequence $\omega \in \Omega$ we define the following sequence:
\begin{equation}
\label{eq:def_sequence}
x_n = \bigcomp_{i=1}^n f_{\omega_i}(x_0) = (f_{\omega_n} \circ \dots \circ f_{\omega_2} \circ f_{\omega_1})(x_0), \quad n \in \mathbb{N}.
\end{equation}
When the role of $\omega$ does not have to be made explicit we will simply denote this sequence by $(x_n)$.

First, we collect some useful observations on the behaviour of the sequence $(x_n)$ for a fixed choice of both the initial point and the symbol sequence.

\begin{lemma}
\label{lemma:properties1}
For any fixed $x_0 \in \mathbb{R}$ and $\omega \in \Omega$ the sequence $(x_n)$ defined in equation~\eqref{eq:def_sequence} satisfies the following properties:
\begin{enumerate}[(i)]
\item If $x_k < 0$ for some $k \in \mathbb{N}_0$, then $x_n \to -\infty$ as $n \to \infty$;
\item If $x_k > 1$ for some $k \in \mathbb{N}_0$, then $x_n \to \infty$ as $n \to \infty$;
\item Precisely one of the following three cases holds:
\begin{itemize}
\item $x_n \to -\infty$ as $n \to \infty$;
\item $x_n \in [0,1]$ for all $n \in \mathbb{N}_0$;
\item $x_n \to +\infty$ as $n \to \infty$.
\end{itemize}
\end{enumerate}
\end{lemma}

\begin{proof}
(i) It suffices to show that $x_n \leq s^{n-k} x_k$ for all $n \geq k$. Clearly, this inequality holds for $n=k$. Now assume that the inequality holds for some $n \geq k$. Noting that $x_n < 0$ and using assumption (A3) gives
\[
f_0(x_n)
	\leq sx_n 
	\leq s^{n+1-k} x_k,
\]
and
\[
f_1(x_n)
	\leq s(x_n-b)
	< sx_n
	\leq s^{n+1-k}x_k.
\]
Since either $x_{n+1} = f_0(x_n)$ or $x_{n+1} = f_1(x_n)$, the desired inequality follows by induction.

(ii) It suffices to show that $x_n \geq 1 + s^{n-k}(x_k-1)$ for all $n \geq k$. Clearly, this inequality holds for $n=k$. Now assume that the inequality holds for some $n \geq k$. Noting that $x_n > 1$ and using assumption (A3) gives
\[
\begin{split}
f_0(x_n)
	& \geq 1 + s(x_n-a) \\
	& = sx_n + 1-sa \\
	& \geq s(1 + s^{n-k}(x_k-1)) + 1-sa \\
	& = 1 + s^{n+1-k}(x_k-1) + (1-a)s \\
	& > 1 + s^{n+1-k}(x_k-1),
\end{split}
\]
and
\[
\begin{split}
f_1(x_n)
	& \geq 1 + s(x_n-1) \\
	& = sx_n + 1-s \\
	& \geq s(1 + s^{n-k}(x_k-1)) + 1-s \\
	& = 1+s^{n+1-k}(x_k-1).
\end{split}
\]
Since either $x_{n+1} = f_0(x_n)$ or $x_{n+1} = f_1(x_n)$, the desired inequality follows by induction.

(iii) This follows from (i) and (ii).
\end{proof}

\begin{lemma}
\label{lemma:properties2}
For any fixed $x_0, y_0 \in \mathbb{R}$ and $\omega \in \Omega$ the sequences $(x_n)$ and $(y_n)$ defined in equation~\eqref{eq:def_sequence} satisfy the following properties:
\begin{enumerate}[(i)]
\item If $x_0 \leq y_0$, then $x_n \leq y_n$ for all $n \in \mathbb{N}_0$;
\item Assume that (A4) holds. If $y_0 = 1-x_0$ but in the definition of $(y_n)$ the symbol sequence $\omega$ is replaced by the flipped sequence $\overline{\omega}$, then $y_n = 1-x_n$ for all $n \in \mathbb{N}$. In particular, $x_n \to \pm\infty$ if and only if $y_n \to \mp \infty$.
\end{enumerate}
\end{lemma}

\begin{proof}
(i) If $x_n \leq y_n$ for some $n \in \mathbb{N}$, then using that the functions $f_0$ and $f_1$ are both increasing gives
\[
x_{n+1} = f_{\omega_n}(x_n) \leq f_{\omega_n}(y_n) = y_{n+1}.
\]
Thus, the claim follows by induction.

(ii) If $y_n = 1-x_n$ for some $n \in \mathbb{N}$, then assumption (A4) implies that
\[
\begin{split}
y_{n+1}
	& = f_{1-\omega_n}(y_n) \\
	& = f_{1-\omega_n}(1-x_n) \\
	& = 1-f_{\omega_n}(x_n) \\
	& = 1-x_{n+1}.
\end{split}
\]
Thus, the claim follows by induction.
\end{proof}

Next, we will consider the sequence $(x_n)$ where the symbol sequence $\omega$ is chosen at random according to the probability measure $\prob_p$. To that end, we need the following result on measurable sets.

\begin{lemma}
\label{lemma:measurable}
For any fixed $x_0 \in \mathbb{R}$ the following sets are measurable:
\begin{enumerate}[(i)]
\item $\{\omega \in \Omega \,:\, x_n \to \infty\}$;
\item $\{\omega \in \Omega \,:\, x_n \to -\infty\}$;
\item $\{\omega \in \Omega \,:\, (x_n) \subseteq [0,1] \}$.
\end{enumerate}

\end{lemma}

\begin{proof}
(i) For fixed $k, t \in \mathbb{N}$ the set
\[
A_{k,t} = \{ \omega \in \Omega \,:\, (x_n) \text{ satisfies } x_k > t \}
\]
is either empty or the union of at most $2^k$ cylinder sets and hence measurable. Since $\sigma$-algebras are closed under countable unions and intersections it follows that
\[
\{\omega \in \Omega \,:\, x_n \to \infty\}
	= \bigcap_{t = 1}^\infty \bigcup_{m=1}^\infty \bigcap_{k=m}^\infty A_{k,t}
\]
is measurable. Statement (ii) is proven similarly.

(iii) Lemma \ref{lemma:properties1}(iii) implies that
\[
\{\omega \in \Omega \,:\, (x_n) \subseteq [0,1] \}
	= \{\omega \in \Omega \,:\, x_n \to \infty\}^c \cap \{\omega \in \Omega \,:\, x_n \to -\infty\}^c
\]
Since $\sigma$-algebras are closed under complements it follows that the right-hand side is measurable.
\end{proof}

\begin{lemma}
\label{lemma:properties3}
For any fixed $x_0 \in \mathbb{R}$ the sequence $(x_n)$ satisfies the following properties:
\begin{enumerate}[(i)]
\item If $x_0 \leq 0$, then $\prob_p(\{\omega \in \Omega \,:\, x_n \to \infty\}) = 0$;
\item If $x_0 \geq 1$, then $\prob_p(\{\omega \in \Omega \,:\, x_n \to \infty\}) = 1$;
\item $\prob_p(\{\omega \in \Omega \,:\, (x_n) \subseteq [0,1]\}) = 0$;
\item $\prob_p(\{\omega \in \Omega \,:\, (x_n) \to -\infty\}) = 1-\prob_p(\{\omega \in \Omega \,:\, (x_n) \to \infty\})$.
\end{enumerate}
\end{lemma}

\begin{proof}
(i) For $x_0 < 0$ this immediately follows from Lemma \ref{lemma:properties1}(i). If $x_0 = 0$, then $x_n = 0$ for all $n \in \mathbb{N}$ if and only if $\omega_n = 0$ for all $n \in \mathbb{N}$; otherwise $x_n < 0$ for some $n \in \mathbb{N}$ and thus $x_n \to -\infty$. In either case, $x_n \to \infty$ cannot occur. 

(ii) For $x_0 > 1$ this immediately follows from Lemma \ref{lemma:properties1}(ii). If $x_1 = 1$, then with the cylinder set $A_n = [1,1,\dots,1]$ of length $n$ we have
\[
\{\omega \in \Omega \,:\, x_n \to \infty\}
	= \{\omega \in \Omega \,:\, \omega_n = 0 \text{ for some } n \in \mathbb{N}\}
	= \Omega \setminus \bigcap_{n=1}^\infty A_n.
\]
Since $\prob_p(A_n) = q^n$ for all $n \in \mathbb{N}$ it follows that the set on the left-hand side has full measure.

(iii) If $x_0 \notin [0,1]$, then Lemma \ref{lemma:properties1} implies that for all choices of $\omega \in \Omega$ we have that $x_n \notin [0,1]$ for all $n \in \mathbb{N}$ and thus $\prob_p(\{\omega \in \Omega \,:\, (x_n) \subseteq [0,1]\}) = 0$ trivially holds. If $x_0 \in [0,1]$, then for all $n \geq 1$ define the set
\[
B_n = \{\omega \in \Omega \,:\, (x_n) \text{ satisfies } x_k \in [0,1] \text{ for all } k=1,\dots,n\}.
\]
Note that $B_n$ can be written as a disjoint union of finitely many cylinder sets: setting
\[
I_n = \{(\omega_1,\dots,\omega_n) \in \{0,1\}^n \,:\, [\omega_1,\dots,\omega_n] \subseteq B_n\}
\]
gives
\[
B_n = \bigcup_{(\omega_1,\dots,\omega_n) \in I_n} [\omega_1,\dots,\omega_n].
\]
Note that $f_1(x_n)<0$ when $x_n \in [0,a]$ and $f_0(x_n)>1$ when $x_n \in (a,1]$. Therefore, setting
\[
\alpha_{n+1} = \begin{cases} 1 & \text{if } 0 \leq x_n \leq a, \\ 0 & \text{if } a < x_n \leq 1, \end{cases}
\]
gives the inclusion
\[
B_{n+1}
	\subseteq B_n \cap \{\omega \in \Omega \,:\, \omega_{n+1} = \alpha_{n+1}\}
	= \bigcup_{(\omega_1,\dots,\omega_n) \in I_n} [\omega_1,\dots,\omega_n, \alpha_{n+1}].
\]
By equation \eqref{eq:cylinder_split} it follows that
\[
\prob_p(B_{n+1})
	\leq \sum_{(\omega_1,\dots,\omega_n) \in I_n} \prob_p([\omega_1,\dots,\omega_n, \alpha_{n+1}])
	= \sum_{(\omega_1,\dots,\omega_n) \in I_n} \prob_p([\omega_1,\dots,\omega_n])\prob_p([\alpha_{n+1}]).
\]
Letting $\mu := \max\{p,q\}$ gives
\[
\prob_p(B_{n+1}) \leq \mu \sum_{(\omega_1,\dots,\omega_n) \in I_n} \prob_p([\omega_1,\dots,\omega_n]) = \mu\prob_p(B_n).
\]
By induction it follows that $\prob_p(B_n) \leq \mu^{n-1}\prob_p(B_1)$ for all $n \in \mathbb{N}$.
Finally, we obtain
\[
\prob_p(\{ \omega \in \Omega \,:\, (x_n) \subseteq [0,1]\})
	= \prob_p\bigg(\bigcap_{n=1}^\infty B_n\bigg)
	= \lim_{n\to\infty} \prob_p(B_n) = 0.
\]

(iv) Lemma \ref{lemma:properties1}(iii) implies that the following union is disjoint:
\[
\Omega =
	\{\omega \in \Omega \,:\, x_n \to -\infty\} \cup
	\{\omega \in \Omega \,:\, (x_n) \subseteq [0,1]\} \cup
	\{\omega \in \Omega \,:\, x_n \to \infty\}.
\]
Taking the measure of both sides and using part (iii) gives
\[
1 = \prob_p(\{\omega \in \Omega \,:\, x_n \to -\infty\}) + \prob_p(\{\omega \in \Omega \,:\, x_n \to \infty\}),
\]
as desired.
\end{proof}


\subsection{The function $F_p$}

We now define the following function:
\[
F_{p} : \mathbb{R} \to \mathbb{R}, \quad
F_{p}(x) = \prob_p(\{ \omega \in \Omega \,:\, \bigcomp_{i=1}^n f_{\omega_i}(x) \to \infty \}).
\]
Lemma \ref{lemma:measurable} ensures that this function is well defined. The dependence of $F_p$ on the choice of the functions $f_0$ and $f_1$ is suppressed in the notation for convenience.

\begin{theorem}
\label{theorem:cantor_properties}
The function $F_{p}$ has the following properties:
\begin{enumerate}[(i)]
\item $F_{p}(x) = 0$ for all $x \leq 0$ and $F_{p}(x) = 1$ for all $x \geq 1$;
\item $F_{p}(x) \leq F_p(y)$ whenever $x \leq y$;
\item $F_p(x) = pF_p(f_0(x)) + qF_p(f_1(x))$ for all $x \in \mathbb{R}$;
\item $F_p(x) = p$ for all $x \in [a,b]$;
\item $F_{p}(1-x) = 1-F_{q}(x)$ for all $x \in \mathbb{R}$ whenever assumption (A4) holds.
\end{enumerate}
\end{theorem}

\begin{proof}
(i) This follows directly from Lemma \ref{lemma:properties3}(i) and (ii).

(ii) If $x \leq y$, then Lemma \ref{lemma:properties2}(i) implies that
\[
\{ \omega \in \Omega \,:\, \bigcomp_{i=1}^n f_{\omega_i}(x) \to \infty \}
	\subseteq
	\{ \omega \in \Omega \,:\, \bigcomp_{i=1}^n f_{\omega_i}(y) \to \infty \},
\]
from which the claim follows.

(iii) Define the sets
\[
A = \{ \omega \in \Omega \,:\, \bigcomp_{i=1}^n f_{\omega_i}(x) \to \infty \}
\quad\text{and}\quad
B_k = \{ \omega \in \Omega \,:\, \bigcomp_{i=1}^n f_{\omega_i}(f_k(x)) \to \infty \},
\]
where $k \in \{0,1\}$. We claim that these sets are related as follows:
\[
A \cap [0] = \sigma^{-1}(B_0) \cap [0]
\quad\text{and}\quad
A \cap [1] = \sigma^{-1}(B_1) \cap [1].
\]
Indeed, we have the following equivalences:
\[
\begin{split}
(\omega_1,\omega_2,\omega_3,\dots) \in \sigma^{-1}(B_0) \cap [0]
	& \quad\Leftrightarrow\quad \omega_1 = 0 \text{ and } (\omega_2,\omega_3,\omega_4,\dots) \in B_0 \\
	& \quad\Leftrightarrow\quad \omega_1 = 0 \text{ and } \bigcomp_{i=2}^{n+1} f_{\omega_i}(f_0(x)) \to \infty \\
	& \quad\Leftrightarrow\quad \omega_1 = 0 \text{ and } \bigcomp_{i=1}^{n} f_{\omega_i}(x) \to \infty \\
	& \quad\Leftrightarrow\quad (\omega_1,\omega_2,\omega_3,\dots) \in A \cap [0].
\end{split}
\]
The other equality follows in a similar manner. Finally, using Lemma \ref{lemma:sigma1} gives
\[
\begin{split}
F_p(x)
	& = \prob_p(A) \\
	& = \prob_p(A \cap [0]) + \prob_p(A \cap [1]) \\
	& = \prob_p(\sigma^{-1}(B_0) \cap [0]) + \prob_p(\sigma^{-1}(B_1) \cap [1]) \\
	& = p\prob_p(B_0) + q\prob_p(B_1) \\
	& = pF_p(f_0(x)) + qF_p(f_1(x)),
\end{split}
\]
as desired.

(iv) Since
\[
f_0(a)=1, \quad
f_1(a)<0, \quad
f_0(b)>1, \quad
f_1(b)=0,
\]
it follows from statements (i) and (iii) that $F_p(a) = F_p(b) = p$. Statement (ii) implies that $F_p(x)=p$ for all $x \in [a,b]$.

(v) Under assumption (A4) we can apply Lemma \ref{lemma:properties2}(ii), which gives
\[
\begin{split}
F_{p}(1-x)
	& = \prob_p(\{ \omega \in \Omega \,:\, \bigcomp_{i=1}^n f_{\omega_i}(1-x) \to \infty \}) \\
	& = \prob_p(\{ \omega \in \Omega \,:\, \bigcomp_{i=1}^n f_{\overline{\omega}_i}(x) \to -\infty \}).
\end{split}
\]
Using Lemma \ref{lemma:properties3}(iv) gives
\[
F_{p}(1-x)
	= 1-\prob_p(\{ \omega \in \Omega \,:\, \bigcomp_{i=1}^n f_{\overline{\omega}_i}(x) \to \infty \}).
\]
Finally, using Lemma \ref{lemma:sigma1} gives
\[
\begin{split}
F_{p}(1-x)
	& = 1-\prob_p(\rec(\{ \omega \in \Omega \,:\, \bigcomp_{i=1}^n f_{\omega_i}(x) \to \infty \})) \\
	& = 1-\prob_q(\{ \omega \in \Omega \,:\, \bigcomp_{i=1}^n f_{\omega_i}(x) \to \infty \}) \\
	& = 1-F_{q}(x),
\end{split}
\]
as desired.
\end{proof}

For any $x \in [0,1]$ we have $f_0^{-1}(x) \in [0,a]$ and $f_1^{-1}(x) \in [b,1]$, which implies $f_1(f_0^{-1}(x)) < 0$ and $f_0(f_1^{-1}(x)) > 1$, and thus
\begin{equation}
\label{eq:Fp-constant}
F_{p}(f_0^{-1}(x)) = pF_{p}(x)
\quad\text{and}\quad
F_{p}(f_1^{-1}(x)) = p + qF_{p}(x).
\end{equation}
Consider the recursion
\begin{equation}
\label{eq:Cn-recursion}
C_{n+1} = f_0^{-1}(C_n) \cup f_1^{-1}(C_n)
\quad\text{where}\quad
C_0 = [a,b].
\end{equation}
Note that every set $C_n$ is a disjoint union of $2^n$ intervals and that in view of equation~\eqref{eq:Fp-constant} the function $F_p$ is constant on each of those intervals. In conclusion, for $a<b$ there exist infinitely many subintervals of $[0,1]$ on which the function $F_{p}$ is constant. The parameters $a$ and $b$ control the width of these intervals and the parameter $p$ controls the constant value that $F_p$ attains on these intervals.

For the purpose of plotting graphs one could take a sufficiently large $k \in \mathbb{N}$ and approximate the value of $F_p(x)$ by the sum of the probabilities $\prob_p ([\alpha_1,\dots,\alpha_k])$ over all $k$-tuples $(\alpha_1,\dots,\alpha_k) \in \{0,1\}^k$ for which $\bigcomp_{i=1}^k f_{\alpha_i}(x) > 1$. Alternatively, one can use the functional equation in Theorem \ref{theorem:cantor_properties}(iii) to derive a sequence of successive approximations as follows. It is well known that the following metric space is complete:
\[
\begin{split}
\mathcal{B} & = \{F : \mathbb{R} \to \mathbb{R} \,:\, F \text{ is bounded}\}, \\
d(F,G) & = \sup_{x \in \mathbb{R}} |F(x) - G(x)|.
\end{split}
\]
Further, observe that the following set is closed in $\mathcal{B}$:
\[
\mathcal{V} = \{ F \in \mathcal{B} \,:\, F|_{(-\infty,0]} \equiv 0 \text{ and } F|_{[1,\infty)} \equiv 1\}.
\]

\begin{lemma}
The map $T : \mathcal{B} \to \mathcal{B}$ defined by
\[
(TF)(x) = pF(f_0(x)) + qF(f_1(x)),
\]
satisfies the following properties:
\begin{enumerate}[(i)]
\item $T(\mathcal{V}) \subseteq \mathcal{V}$;
\item $d(TF, TG) \leq \max\{p,q\}\,d(F,G)$ for all $F, G \in \mathcal{V}$.
\end{enumerate}
\end{lemma}

\begin{proof}
(i) This follows from assumptions (A1)--(A3).

(ii) For all $x \in \mathbb{R}$ we have
\[
(TF)(x)-(TG)(x) = 
\begin{cases}
p[F(f_0(x))-G(f_0(x))]	& \text{if } x \in [0,a], \\
q[F(f_1(x))-G(f_1(x))]	& \text{if } x \in [b,1], \\
0									& \text{otherwise},
\end{cases}
\]
from which the claim readily follows.
\end{proof}

The Contraction Mapping Theorem implies that $T$ has a unique fixed point in $\mathcal{V}$; in fact, by Theorem \ref{theorem:cantor_properties} it follows that this fixed point is the function $F_p$. Now consider the following iteration scheme:
\[
F_{p,k} = T^k F_{p,0}
\quad\text{where}\quad
F_{p,0}(x) = 
\begin{cases}
0 & \text{if } x \leq 0, \\
x & \text{if } 0 < x < 1, \\
1 & \text{if } x \geq 1.
\end{cases}
\]
An induction argument shows that
\[
d(F_p, F_{p,k}) \leq \max\{p,q\}^k d(F_p, F_{p,0}) \leq \max\{p,q\}^k,
\]
where the last inequality uses the fact that both $F_p$ and $F_{p,0}$ attain their values in $[0,1]$ and hence $d(F_p,F_{p,0}) \leq 1$. In particular, $F_{p,k} \to F_p$ uniformly on $\mathbb{R}$; since all functions $F_{p,k}$ are continuous, it follows that $F_p$ is continuous as well. Since the convergence is exponentially fast, $F_{p,k}$ is an accurate approximation of $F_p$ even for moderate values of $k$.


\newpage

\section{Examples of singular functions}
\label{sec:cantor}

For particular choices of the functions $f_0$ and $f_1$ it follows that $F_p$ equals a singular function. Examples include Cantor-like functions, Lebesgue singular functions, and the Minkowski question mark function. We show how these examples can be obtained for explicit choices of $f_0$ and $f_1$.


\subsection{Cantor functions}

For $0 < a \leq b < 1$ consider the functions
\begin{equation}
\label{eq:f-example1}
f_0(x) = x/a
\quad\text{and}\quad
f_1(x) = (x-b)/(1-b).
\end{equation}
Clearly, these functions satisfy assumptions (A1)--(A3) with $s=\min\{1/a,1/(1-b)\}$. For this particular choice of $f_0$ and $f_1$ the function $F_p$ belongs to a family of Cantor-like functions. Figure \ref{fig:plot1} shows the graph of $F_{p}$ for two choices of the parameters $(a, b, p)$. The symmetry assumption (A4) holds if and only if $b=1-a$; Cantor functions that satisfy this symmetry condition play a role in extreme value laws for certain autoregressive processes~\cite{S:2025}. 

Since $f_0^{-1}(x) = ax$ and $f_1^{-1}(x) = (1-b)x + b$ are linear functions, it follows from equation~\eqref{eq:Cn-recursion} that the Lebesgue measure of the sets $C_n$ satisfies the following recursion:
\[
\Leb(C_{n+1})
	= a\Leb(C_n) + (1-b)\Leb(C_n)
	= (1-(b-a))\Leb(C_n).
\]
If $a<b$, then the total length of all intervals on which $F_p$ is constant is given by
\[
\Leb\bigg(\bigcup_{n=0}^\infty C_n\bigg) = (b-a)\sum_{n=0}^\infty (1-(b-a))^n = 1,
\]
which implies that $F_p$ has zero derivative almost everywhere.

\begin{figure}[h]
\centering
\includegraphics[width=0.45\textwidth]{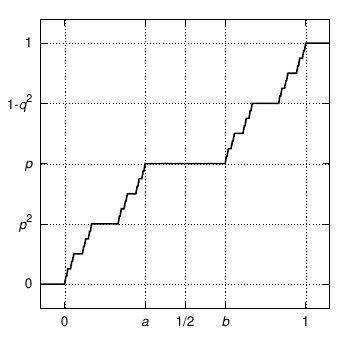}
\includegraphics[width=0.45\textwidth]{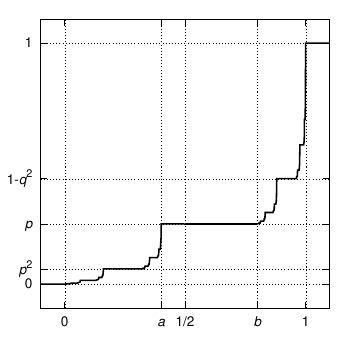}
\caption{Graphs of $F_p$ for the functions $f_0$ and $f_1$ as in equation \eqref{eq:f-example1}. Left: $(a,b,p)=(1/3,2/3,1/2)$ which is the classical Cantor function. Right: $(a,b,p)=(2/5,4/5,1/4)$.}
\label{fig:plot1}
\end{figure}

For the particular choice $a=1/3$, $b=2/3$, and $p=1/2$ the function $F_p$, when restricted to the unit interval, is indeed the classical Cantor function as defined in equation \eqref{eq:cantor_function}. This follows from the following characterization proven by Chalice \cite{Chalice:91}.

\begin{theorem}
\label{theorem:chalice}
Assume that a function $F : [0,1] \to \mathbb{R}$ satisfies the following properties:
\begin{enumerate}[(i)]
\item $F$ is non-decreasing;
\item $F(x/3) = F(x)/2$ for all $x \in [0,1]$;
\item $F(1-x) = 1 - F(x)$  for all $x \in [0,1]$.
\end{enumerate}
Then $F$ is the Cantor function.
\end{theorem}

Indeed, properties (i) and (iii) immediately follow from Theorem \ref{theorem:cantor_properties}. For all $y \in \mathbb{R}$ we have
\[
F_{1/2}(y) = \frac{1}{2}F_{1/2}(3y) + \frac{1}{2} F_{1/2}(3y-2),
\]
and by setting $y = x/3$ with $x \in [0,1]$ we obtain
\[
F_{1/2}(x/3) = F_{1/2}(x)/2,
\]
which shows that property (ii) is satisfied as well.


\subsection{Lebesgue singular functions}

\begin{figure}[h]
\centering
\includegraphics[width=0.45\textwidth]{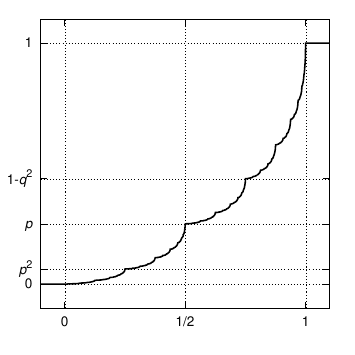}
\includegraphics[width=0.45\textwidth]{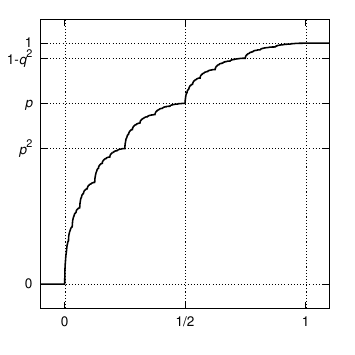}
\caption{Graphs of $F_p$ for the functions $f_0$ and $f_1$ as in equation \eqref{eq:f-example1} with $c=1/2$ and $p=1/4$ (left) and $p=3/4$ (right).}
\label{fig:plot2}
\end{figure}

Taking $a=b=1/2$ in equation \eqref{eq:f-example1} implies that $F_p$ satisfies the functional equation
\[
F_p(x) = p F_p(2x) + q F_p(2x-1).
\]
It is easy to verify that for $p=1/2$ the piecewise linear function
\[
F_{1/2}(x) =
\begin{cases}
0 & \text{if } x < 0, \\
x & \text{if } 0 \leq x \leq 1, \\ 
1 & \text{if } x > 1,
\end{cases}
\]
satisfies the above functional equation. In contrast, for $p \neq 1/2$ we obtain Lebesgue's singular function \cite{Sumi:2011}. In this case $F_p$ is strictly increasing and has no intervals on which the function is constant; see \cite{Kawamura:2011} and references therein. For two choices of $p$ the graph of $F_p$ is shown in Figure \ref{fig:plot2}.


\subsection{Minkowski's question mark function}

Consider the functions
\begin{equation}
\label{eq:f-example3}
\begin{split}
f_0(x)
	& =
\begin{cases}
x/(1-x)	& \text{if } 0 \leq x \leq 1/2, \\
2x		& \text{otherwise},
\end{cases} \\
f_1(x) & = 1-f_0(1-x).
\end{split}
\end{equation}
These functions satisfy assumptions (A1)--(A3) with $s=2$. By construction, assumption (A4) is trivially satisfied. For two choices of $p$ the graph of $F_p$ is shown in Figure \ref{fig:plot3}.

In the particular case $p=1/2$ it follows from Theorem \ref{theorem:cantor_properties}(iii) and (iv) we obtain the functional equations
\[
F_{1/2}(1-x) = 1-F_{1/2}(x)
\quad\text{for all}\quad
x \in \mathbb{R},
\]
and
\[
F_{1/2}(x/(x+1)) = F_{1/2}(x)/2
\quad\text{for all}\quad
x \in [0,1],
\]
Since $F_{1/2}$ is bounded as well, it follows that the restriction of $F_{1/2}$ to the unit interval is in fact the Minkowski question mark function \cite{Girgensohn:1996, Kairies:1997} which was originally introduced as an example of continuous and monotone function that maps rationals to dyadic rationals and quadratic irrationals to non-dyadic rationals.

\begin{figure}[h]
\centering
\includegraphics[width=0.45\textwidth]{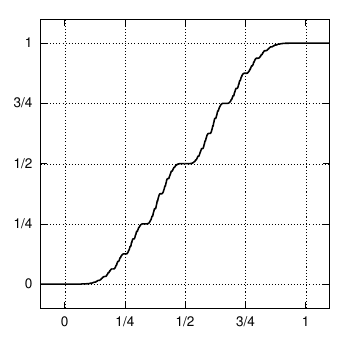}
\includegraphics[width=0.45\textwidth]{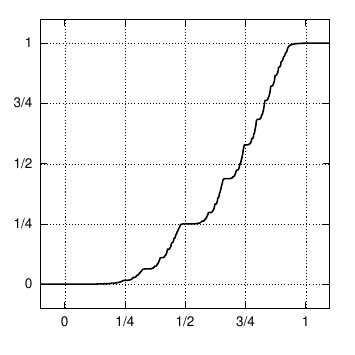}
\caption{Graphs of $F_p$ for the functions $f_0$ and $f_1$ as in equation \eqref{eq:f-example3} $p=1/2$ (left) and $p=1/4$ (right).}
\label{fig:plot3}
\end{figure}


\newpage

\section{Suggestions for further research}
\label{sec:discussion}

An obvious extension of the results above would be to consider random iteration with finitely many functions. Armstrong and Schaubroeck~\cite{ArmstrongSchaubroeck:2020} explored this idea for linear functions $f_k(x) = a_k x - b_k$ where $a_k$ and $b_k$ are non-negative integers that satisfy appropriate conditions. Their paper provides a full explanation on which intervals the resulting generalized Cantor function is constant, but only the case of linear functions and equal probabilities is considered.

Okamoto \cite{Okamoto:2005} introduced a one-parameter family of self-affine functions on the unit interval as the unique continuous solution of the functional equation
\[
F_a(x)
=
\begin{cases}
a F_a(3x)			& \text{if } 0 \leq x \leq 1/3, \\
(1-2a)F_a(3x-1) + a	& \text{if } 1/3 < x \leq 2/3, \\
a F_a(3x-2) + 1-a	& \text{if } 2/3 < x \leq 1,
\end{cases}
\]
where $0 < a < 1$. For $a=1/2$ the classical Cantor function is obtained, and for $a=5/6$ one obtains Perkins' continuous but nowhere differentiable function. An interesting question is whether all the members of this family can be obtained from random function iteration. The functional equation suggests that perhaps at least three functions would be needed.

Subshifts of finite type \cite{Berger:2001} provide a general framework for random iteration that also allows restrictions on the order in which the functions are chosen. Let $V$ be a set of $k$ distinct symbols. Since the particular choice of the labels of these symbols is irrelevant one can simply take $V=\{1,\dots,k\}$. Let $A$ be a $k \times k$ adjacency matrix with entries in $\{0,1\}$ and consider the following set of infinite sequences:
\[
\Omega_A
	= \big\{(\omega_1,\omega_2,\omega_3,\dots) \,:\, \omega_i \in V \text{ and } A_{\omega_i, \omega_{i+1}} = 1 \text{ for all } i \geq 1\big\}.
\]
Thus, the matrix $A$ encodes which symbol transitions are allowed. Take a stochastic $k\times k$ matrix $P$ such that $P_{i,j}>0$ if and only if $A_{i,j}=1$ and a $1 \times k$ probability vector $\pi$ such that $\pi P = \pi$ and define for each cylinder set
\[
\prob_{\pi, P}([\alpha_1,\alpha_2\dots,\alpha_n]) = \pi_{\alpha_1} P_{\alpha_1,\alpha_2}\dots P_{\alpha_{n-1},\alpha_n}.
\]
One can extend this to a so-called Markov measure on the $\sigma$-algebra generated by the cylinder sets. In addition, this measure is invariant under the one-sided shift. In this way, we have a framework for iterating finitely many functions with restrictions. It is then of interest what kind of singular functions can arise from this and how their properties depend on the choice of the functions $f_1,\dots,f_k$, the matrices $A$ and $P$, and the vector $\pi$.


\end{document}